\documentclass{amsart}

\usepackage[latin1]{inputenc}
\usepackage{amsfonts}
\usepackage{amsmath}
\usepackage{amsthm}
\usepackage{amssymb}
\usepackage{latexsym}
\usepackage{enumerate}

\newtheorem{theorem}{Theorem}%[section]
\newtheorem{lemma}[theorem]{Lemma}
\newtheorem{proposition}[theorem]{Proposition}

\newtheorem{question}[theorem]{Question}

\newtheorem{definition}[theorem]{Definition}

\numberwithin{equation}{section}

\DeclareMathOperator{\lin}{lin}
\DeclareMathOperator{\dens}{dens}

\DeclareMathOperator{\cf}{cf}
\newtheoremstyle{TheoremNum}
{\topsep}{\topsep}              %%% space between body and thm
{\itshape}                      %%% Thm body font
{}                              %%% Indent amount (empty = no indent)
{\bfseries}                     %%% Thm head font
{.}                             %%% Punctuation after thm head
{ }                             %%% Space after thm head
{\thmname{#1}\thmnote{ \bfseries #3}}%%% Thm head spec
\theoremstyle{TheoremNum}
\newtheorem{reptheorem}{Theorem}
\newtheoremstyle{TheoremNum}
{\topsep}{\topsep}              %%% space between body and thm
{\itshape}                      %%% Thm body font
{}                              %%% Indent amount (empty = no indent)
{\bfseries}                     %%% Thm head font
{.}                             %%% Punctuation after thm head
{ }                             %%% Space after thm head
{\thmname{#1}\thmnote{ \bfseries #3}}%%% Thm head spec
\theoremstyle{TheoremNum}

\begin{document}
	
	\newcommand{\cc}{\mathfrak{c}}
	\newcommand{\N}{\mathbb{N}}
	\newcommand{\BB}{\mathbb{B}}
	\newcommand{\C}{\mathbb{C}}
	\newcommand{\Q}{\mathbb{Q}}
	\newcommand{\R}{\mathbb{R}}
	\newcommand{\Z}{\mathbb{Z}}
	\newcommand{\T}{\mathbb{T}}
	\newcommand{\st}{*}
	\newcommand{\PP}{\mathbb{P}}
	\newcommand{\rin}{\right\rangle}
	\newcommand{\SSS}{\mathbb{S}}
	\newcommand{\forces}{\Vdash}
	\newcommand{\dom}{\text{dom}}
	\newcommand{\osc}{\text{osc}}
	\newcommand{\F}{\mathcal{F}}
	\newcommand{\h}{\mathcal{H}}
	\newcommand{\A}{\mathcal{A}}
	\newcommand{\B}{\mathcal{B}}
	\newcommand{\I}{{\mathcal{H}_\sigma}}
	\newcommand{\X}{\mathcal{X}}
	\newcommand{\Y}{\mathcal{Y}}
	\newcommand{\CC}{\mathcal{C}}
	\newcommand{\non}{\mathfrak{non}}
	\newcommand{\add}{\mathfrak{add}}
	\newcommand{\cov}{\mathfrak{cov}}
	\newcommand{\cof}{\mathfrak{cof}}
	
	\author{Damian G\l odkowski}
	\address{Faculty of Mathematics, Informatics, and Mechanics, 
		University of Warsaw, ul. Banacha 2, 02-097 Warszawa, Poland}
	\email{\texttt{d.glodkowski@uw.edu.pl}}

	\author{Piotr Koszmider}
	\address{Institute of Mathematics of the Polish Academy of Sciences,
		ul.  \'Sniadeckich 8,  00-656 Warszawa, Poland}
	\email{\texttt{piotr.koszmider@impan.pl}}
	
	\thanks{The authors were partially supported by the NCN (National Science
		Centre, Poland) research grant no.\ 2020/37/B/ST1/02613.}


	%\subjclass[2010]{}
	\title{On coverings of Banach spaces and their subsets by  hyperplanes}

	\begin{abstract} 
		Given a Banach space we consider the $\sigma$-ideal of all of 
		its subsets  which are covered by 
		countably many hyperplanes and investigate its standard cardinal characteristics
		as the additivity, the covering number, the uniformity, the cofinality.
		We determine their values for separable Banach spaces, and 
		approximate them for  nonseparable Banach spaces. The remaining questions 
		reduce to deciding if the following can be proved in ZFC for every nonseparable Banach space $X$:
		(1) $X$ can be  covered by $\omega_1$-many
		of its hyperplanes; (2)  All subsets of $X$ of  cardinalities less than $\cf([\dens(X)]^\omega)$
		can    be covered by countably many hyperplanes.
		We prove (1) and (2) for all Banach spaces in many well-investigated classes and that they
		are consistent with any possible size of the continuum. (1) is related to the problem whether every compact Hausdorff space which has small diagonal is metrizable and (2) to  large cardinals.
		
	\end{abstract}
	
	\maketitle
	
	\section{Introduction}
	
	All Banach spaces considered in this paper are of dimension bigger than $1$ and over the reals. 
	For unexplained terminology see Section \ref{terminology}.
	A hyperplane of a Banach space $X$ is a one-codimensional closed subspace of $X$. It is easy to see that
	it is nowhere dense in $X$. The
	family of all hyperplanes of $X$ will be denoted by $\h(X)$.
	Given a Banach space $X$  one can define the hyperplane ideal
	$\I$ of $X$ as
	$$\I(X)=\{Y\subseteq X: \exists \F\subseteq \h(X) \ \  Y\subseteq \bigcup \F, \ \F\ \hbox{countable}\}.$$
	That is, $\I(X)$ is the family of all subsets
	of $X$ which can be covered by countably many hyperplanes of $X$.
	By the Baire category theorem $X\not\in \I(X)$ for any Banach space $X$.
	We consider the standard  cardinal characteristics of the ideal $\I(X)$:  
	\begin{itemize}
		\item $\add(X)$ is the minimal cardinality of a family of sets from $\I(X)$ whose union is not in $\I(X)$,
		\item $\cov(X)$ is the minimal cardinality of a family of sets from $\I(X)$ whose union is equal to $X$,
		\item $\non(X)$ is the minimal cardinality of a subset of $X$ that is not in $\I(X)$,
		\item $\cof(X)$ is the minimal cardinality of a family of sets from $\I(X)$ such that each member of $\I(X)$ is contained in some element of that family. 
	\end{itemize}
	Such cardinal characteristics are standard tools for investigating  the combinatorial properties of
	a $\sigma$-ideal. The most known case are their applications to the understanding of the
	ideal of Lebesgue measure zero sets and
	the ideal of meager sets of the reals (see e.g. \cite{bartoszynski}). It is easy to observe that
	if the ideal  is proper and  contains all singletons we have 
	the following inequalities: $\add \leq \cov \leq \cof$ and  $\add \leq \non \leq \cof$. 
	The purpose of this paper is to investigate the possible values of the above cardinals for
	the ideal $\I(X)$ and understand how they depend on $X$. A somewhat surprising conclusion is
	that the values depend almost entirely only on the density of $X$ and the $X^*$ or even are fixed for 
	all separable and all nonseparable Banach spaces.
	The first result presented in Section \ref{separable} describes these values for all separable Banach spaces.
	It is an immediate consequence of appropriately formulated result from \cite{klee}:
	
	\begin{theorem}\label{main-separable} Suppose that $X$ is a separable Banach space
		of dimension bigger than $1$. Then
		the following equalities hold:
		\begin{itemize}
			\item $\mathfrak{add}(X)=\omega_1$, 
			\item $\mathfrak{non}(X)=\omega_1$, 
			\item $\mathfrak{cov}(X)=\mathfrak{c}$,
			\item $\mathfrak{cof}(X)=\mathfrak{c}$.
		\end{itemize}
	\end{theorem}
	In fact the values of $\mathfrak{add}$ and $\mathfrak{cof}$ are always trivial 
	(Propositions \ref{basic-add}, \ref{basic-cof}) due to an elementary fact that
	$H\subseteq G$ implies $H=G$ for any two $G, H\in \h(X)$ any $X$ (Proposition \ref{inclusion}).
	The results from Section \ref{basics} provide also much information about the general case including the nonseparable case:
	\begin{theorem}\label{main-nonseparable} Suppose that $X$ is a
		Banach  space of dimension bigger than $1$. Then
		the following equalities and inequalities hold:
		\begin{itemize}
			\item $\mathfrak{add}(X)=\omega_1$, 
			\item $\omega_1\leq\mathfrak{cov}(X)\leq\mathfrak{c}$,
			\item $\dens(X)\leq \mathfrak{non}(X)\leq \cf([\dens(X)]^\omega)$, 
			\item $\mathfrak{cof}(X)=|X^*|$.
		\end{itemize}
	\end{theorem}
	\begin{proof} Propositions \ref{basic-add}, \ref{basic-cof}, \ref{basic-non}, \ref{basic-cov}.
	\end{proof}
	So the interesting cardinal characteristics are $\cov$ and $\non$.
	First we note that making additional (but diverse)
	set theoretic assumptions (which are known to be undecidable) 
	the values of $\cov$ and $\non$ are 
	completely determined by the density of the space or even fixed. For $\non$ this follows just from
	results on cardinal arithmetic and Theorem \ref{main-nonseparable}:
	\begin{theorem}\label{main-consistency} Assume the Generalized Continuum Hypothesis {\sf GCH} or Martin's Maximum {\sf MM}. Let $X$ be a nonseparable Banach space. Then
		\begin{enumerate}
			\item $\mathfrak{cov}(X)=\omega_1$,
			\item $\mathfrak{non}(X)=\dens(X)$ \ \ if $\cf(\dens(X))>\omega$,
			\item $\mathfrak{non}(X)=\dens(X)^+$ \ \ if $\cf(\dens(X))=\omega$.
		\end{enumerate}
		Moreover the same is consistent with any possible size of the continuum $\mathfrak c$.
		If  violations of the above equalities concerning $\non$ are consistent, then so
		is the existence of a measurable cardinal.
	\end{theorem}
	\begin{proof}
		Propositions \ref{basic-cov}, \ref{con-cov},  \ref{non-gch-mm}, \ref{non-anyc},
		\ref{measurable} and the fact that 
		{\sf MM} implies {\sf PFA}.
	\end{proof}
	Not only consistent set-theoretic hypotheses determine the values of $\cov$.
	Also 
	a well-known topological statement  which is unknown to be provable but known to be consistent 
	fixes the value of $\cov$.
	
	\begin{theorem}\label{main-diagonal} Assume that all compact Hausdorff spaces with small diagonal are metrizable.
		Let $X$ be a nonseparable Banach space. Then
		$$\mathfrak{cov}(X)=\omega_1.$$
	\end{theorem}
	\begin{proof} Lemma \ref{dualball-diagonal}.
	\end{proof}
	For the definition of a space with small diagonal see Definition \ref{def-small-diagonal}. In fact a weaker natural topological hypothesis  has the same impact on $\cov$ 
	(see Question \ref{compact-omega1}). The following main questions remain open:
	
	\begin{question}\label{main-question} Can one prove in ZFC any of  the following sentences?
		\begin{enumerate}
			\item Every nonseparable Banach space can be covered by $\omega_1$ of its hyperplanes.
			\item In any Banach space $X$ of dimension bigger than $1$ each subset of cardinality smaller than
			$\cf([\dens(X)]^\omega)$ can be covered by countably many hyperplanes.
		\end{enumerate}
	\end{question}
	The positive answer to the above questions would settle the values of $\cov$ an $\non$ in ZFC
	as in Theorem \ref{main-consistency}.
	Note that by Theorem \ref{main-nonseparable} (3) in every infinite dimensional Banach space $X$ there
	is a subset of cardinality $\cf([\dens(X)]^\omega)$ which cannot be covered by countably many hyperplanes.
	Attempting to prove the sentences of Question \ref{main-question} for all nonseparable Banach spaces  we manage to prove them
	in many cases:
	
	\begin{theorem}\label{main-cov} Suppose that $X$ is any nonseparable Banach space belonging to one
		of the following classes:
		\begin{enumerate}
			\item $X$ admits a fundamental biorthogonal system,
			\item $X$ is of the form $C(K)$  for $K$ scattered, Hausdorff compact,
			\item $X$ contains an isomorphic copy of $\ell_1(\omega_1)$,
			\item The dual ball $B_{X^*}$ of $X^*$ has uncountable tightness in the weak$^*$ topology.
		\end{enumerate}
		Then  $X$
		can be covered by $\omega_1$ hyperplanes, i.e., $\mathfrak{cov}(X)=\omega_1$.  
	\end{theorem}
	\begin{proof}
		Propositions \ref{fund-biort}, \ref{scattered},
		\ref{ell1-inside},  Lemmas \ref{small-diagonal}, \ref{dualball-diagonal}.
	\end{proof}
	Note that this implies that spaces like $c_0(\kappa)$, $\ell_p(\kappa)$ for $1\leq p<\infty$, 
	and any $\kappa>1$,
	reflexive spaces, WLD spaces (by (1)), $\ell_\infty(\kappa)$, $L_\infty(\{0,1\}^\kappa)$ for any $\kappa>1$,
	(by (3)) satisfy the conclusion of the above theorem.
	
	\begin{theorem}\label{main-non} Suppose that $X$ is any  Banach space of
		dimension bigger than $1$ belonging to one
		of the following classes:
		\begin{enumerate}
			\item $X$ admits a fundamental biorthogonal system,
			\item $X$ has density $\omega_n$ for some $n\in \N$.
		\end{enumerate} 
		Then
		each subset of $X$ of cardinality smaller than
		$\cf([\dens(X)]^\omega)$ can be covered by countably many hyperplanes, i.e.,
		$\non(X)=\cf([\dens(X)]^\omega)$.
	\end{theorem}
	\begin{proof}
		Propositions \ref{non-omegan} and \ref{fund-non}.
	\end{proof}
	Note that (1) above implies that spaces like $c_0(\kappa)$, $\ell_p(\kappa)$ for $1\leq p<\infty$, 
	spaces $\ell_\infty(\kappa)$, $L_\infty(\{0,1\}^\kappa)$ 
	for any $\kappa>1$ (by a result of \cite{davis} since
	$\ell_2(dens(X))$ is a quotient of such spaces $X$),
	reflexive spaces, WLD spaces 
	satisfy the conclusion of the above theorem.
	Note that a Banach space $X$ of density $\omega_n$ for $n\in \N$ may have
	cardinality arbitrarily bigger than $\omega_n$ as
	$|X|=\dens(X)^\omega=\omega_n^\omega=\mathfrak c\cdot\omega_n$ by Proposition \ref{card-banach}.
	
	Let us also note one application of our results. Recall that a subset $Y$ of a Banach space
	$X$ is overcomplete (\cite{russo}, \cite{over}) if $|Y|=\dens(X)$ and every subset $Z\subseteq Y$ of cardinality $\dens(X)$  is linearly dense in $X$. The following constitutes a progress on Question 39 from \cite{over}.

	\begin{theorem} Assume the Proper Forcing Axiom {\sf PFA}. Let $X$ be a Banach
		space such that $\cf(\dens(X))>\omega_1$. Then $X$ does not admit an overcomplete set.
		Moreover this statement is consistent with any possible size of the continuum $\mathfrak c$.
	\end{theorem}
	\begin{proof}
		By Theorem \ref{main-consistency} the hypothesis implies that every nonseparable
		Banach space $X$ can be covered by $\omega_1$ many hyperplanes
		$\{H_\alpha: \alpha<\omega_1\}$. If $Y\subseteq X$ and $|Y|=\dens(X)$, then by $\cf(\dens(X))>\omega_1$ there is
		$\alpha<\omega_1$ such that $|H_\alpha\cap Y|=\dens(X)$, so $Z=H_\alpha\cap Y$ witnesses that
		$Y$ is not overcomplete.
	\end{proof}
	
	The structure of the paper is the following. Section 2 contains preliminaries. Section \ref{basics}
	establishes Theorems \ref{main-separable} and \ref{main-nonseparable}. Section \ref{cov} includes
	progress on Question \ref{main-question} (1) and arrives at Theorems \ref{main-consistency} (1),
	\ref{main-diagonal} and
	\ref{main-cov}. Section \ref{non} includes
	progress on Question \ref{main-question} (2) and arrives at Theorems \ref{main-consistency} (2), (3)
	and \ref{main-non}. The last Section \ref{final} discusses the perspectives for further research
	and states additional questions.
	
	No knowledge of logic or higher set-theory is required from the reader to follow the paper.
	This is because all consistency results are obtained by applying consistency results already
	present in the literature.
	
	\section{Preliminaries}
	
	\subsection{Notation and terminology}\label{terminology}
	We use notation common for set theory and the theory of Banach spaces. 
	
	$\N$ denotes the set of non-negative integers. $\Q$ and $\R$ denote the rationals and the reals respectively. $|A|$ denotes the cardinality of a set $A$. By $f|A$ we mean the restriction of a function $f$ to a set $A$. $\omega_n$ stands for the $n$-th infinite cardinal, $\omega_\omega$ is the smallest cardinal which is greater than $\omega_n$ for each $n\in \N$. $\mathfrak c$ denotes the cardinality of $\R$ and is called the continuum. If $\alpha$ is an ordinal number, then $\cf(\alpha)$ denotes its cofinality. $[A]^\omega$ is the set of countable subsets of $A$, $\cf([A]^\omega)$ denotes the cofinality of $[A]^\omega$ considered as the set partially ordered by inclusion, that is the minimal cardinality of a family of countable
	subsets of $A$ such that any countable subset of $A$ is included in an element of the family. 
	
	For a Banach space $X$, density $\dens(X)$ is the minimal cardinality of a dense subset in $X$ (in the norm topology). $X^*$ stands for the Banach space of bounded linear functional on $X$. For $S\subseteq X$ by $\lin(S)$ we denote the smallest linear subspace of $X$ containing $S$ and $\overline{\lin(S)}$ stands for its closure. $\ker (x^*)$ denotes the kernel of a functional $x^*\in X^*$. If $x_i\in X, x_i^*\in X^*$ for $i\in I$ are such that $x^*_i(x_j)=\delta_{i,j}$, then $(x_i,x^*_i)_{i\in I}$ is called a biorthogonal system. If moreover $X=\overline{\lin\{{x_i:i\in I}\}}$, then such a system is called fundamental. For the definition and various characterizations of WLD spaces see \cite{wld}.
	
	We assume that all considered topological spaces are Hausdorff. For a compact space $K$ we denote by $C(K)$ the Banach space of real continuous functions on $K$ with the supremum norm. If $x\in K$, then $\delta_x\in C(K)^*$ is defined by $\delta_x(f)=f(x)$. $\Delta(K)=\{(x,x): x\in K\}$ is the diagonal of $K\times K$. 
	
	For any set $A$ by $c_0(A)$ we denote the Banach space of functions $f:A\rightarrow \R$ such that for each $\varepsilon>0$ there is finitely many $a\in A$ with $|f(a)|>\varepsilon$ with the supremum norm. For $1\leq p< \infty $ by $\ell_p(A)$ we denote the Banach space of functions $f:A\rightarrow \R$ such that $\|f\|^p=\sum_{a\in A} |f(a)|^p<\infty$. $\ell_\infty(A)$ stands for the Banach space of bounded functions $f:A\rightarrow \R$ with the supremum norm. $\ell_\infty^c(A)$ is the subspace of $\ell_\infty(A)$ consisting of functions with countable support. We also write $c_0(\N)=c_0, \ell_p(\N)=\ell_p$ and $\ell_\infty(\N)=\ell_\infty$. 
	
	ZFC denotes Zermelo-Fraenkel set theory with the axiom of choice. We say that a sentence $\varphi $ is relatively consistent with a set of axioms if 
	its negation $\neg \varphi$ cannot be proven from those axioms unless assuming ZFC leads to contradiction.
	We usually skip the word ``relatively". A sentence $\varphi$ is undecidable if both $\varphi, \neg \varphi$ are  consistent. {\sf CH} means `$\mathfrak c= \omega_1$'. {\sf MM} stands for Martin's Maximum and {\sf PFA} for Proper Forcing Axiom. It is known that {\sf MM} implies {\sf PFA} and {\sf PFA} implies $\mathfrak c=\omega_2$ (for the definitions of {\sf MM} and {\sf PFA} and proofs of mentioned facts check \cite{jech}).

	\subsection{Hyperplanes}
	Let us recall here some elementary and well-known facts concerning hyperplanes in Banach spaces.

	\begin{lemma}\label{inclusion} Suppose that $X$ is a Banach space. Then the following hold.
		\begin{enumerate}
			\item If  $H, G$ are hyperplanes of $X$ and $H\subseteq G$,
			then $H=G$.
			\item If a hyperplane $H$ is contained in a countable union $\bigcup_{i\in \N} H_i$ of hyperplanes $H_i$, 
			then $H=H_i$ for some $i\in \N$.
		\end{enumerate}
	\end{lemma}
	\begin{proof} (1) Every hyperplane in a Banach space $X$ is a kernel of some non-zero bounded functional and kernels of $f,g\in X^*$ are different if and only if $f$ and $g$ are linearly independent (3.1.13, 3.1.14 of
		\cite{semadeni}). 
		
		For (2) assume that $H\not\subseteq H_i$ for any $i\in\omega$. Then $H_i\cap H$ are nowhere dense in $H$. Hence by Baire category theorem $\bigcup_{i\in \omega} H_i\cap H$ has empty interior in $H$, which leads to contradiction with $H=\bigcup_{i\in \omega} H_i\cap H$. Now use (1).
	\end{proof}

	\subsection{Cardinalities of Banach spaces}
	
	Let us recall here some well-known facts concerning cardinalities of Banach spaces.
	The first one follows from the Lemma 2.8 of \cite{banach-cardinality} and the fact that $(\kappa^\omega)^\omega=\kappa^\omega$. 
	\begin{proposition}\label{card-banach}
		If $X$ is a Banach space, then $\dens(X)^\omega=|X|^\omega=|X|$.
	\end{proposition}

	\begin{proposition}\label{card-dual}
		If $X$ is a Banach space of dimension bigger than $1$, then $|X^*|=|\h(X)|$.
	\end{proposition}
	\begin{proof}
		Every hyperplane in a Banach space $X$ is a kernel of some non-zero bounded functional and kernels of $f,g\in X^*$ are different if and only if $f$ and $g$ are linearly independent (3.1.13, 3.1.14 of
		\cite{semadeni}). So  $|X^*|=\mathfrak{c}\cdot|\h(X)|$. If $f,g\in X^*$ are linearly independent, then the kernels of $f+\lambda g$  are different for different choices of $\lambda \in \R\backslash \{0\}$. So
		$\mathfrak{c}\leq|\h(X)|$ and so $|X^*|=|\h(X)|$.
	\end{proof}
	
	Note that $|X^*|$ is not determined by $|X|$ or $\dens(X)$. By  Proposition \ref{card-banach}
	we have $\dens(c_0(\mathfrak{c}))=|c_0(\mathfrak{c})|=\dens(\ell_1(\mathfrak{c}))=
	|\ell_1(\mathfrak{c})|=\mathfrak{c}$ and $\dens(\ell_\infty(\mathfrak{c}))=
	|\ell_\infty(\mathfrak{c})|=2^{\mathfrak{c}}$ while $c_0(\mathfrak{c})^*=\ell_1(\mathfrak{c})$ 
	and $\ell_1^*(\mathfrak{c})=\ell_\infty(\mathfrak{c})$.
	
	\subsection{Ideals}
	\begin{proposition}\label{simple-cichon} Let $X$ be a Banach space of dimension bigger than $1$. Then 
		$\add(X) \leq \cov(X) \leq \cof(X)$ and  $\add (X)\leq \non(X) \leq \cof(X)$. 
	\end{proposition}
	\begin{proof} This is elementary. Since $\I(X)$ contains all singletons and is a $\sigma$-ideal,
		Lemma 1.3.2 of \cite{bartoszynski} applies.
	\end{proof}
	
	\section{Basic results on the values of the cardinal characteristics}\label{basics}
	
	\subsection{Separable Banach spaces}\label{separable}
	
	It turns out that the values of our cardinal characteristics on separable Banach spaces are 
	the same. We include the proof of the following result for the convenience of the reader.
	
	\begin{proposition}[{\cite[Theorem 2.4]{klee}}]\label{klee} 
		Let $X$ be a separable Banach space. Then there exists a set $Y\subseteq X $ of cardinality $\mathfrak c$ such that for every hyperplane $H$ of $X$ the set $H\cap Y$ is finite.
	\end{proposition}
	\begin{proof} 
		Let 
		$\{x_n: n\in \N\}\subseteq X$ be linearly dense in $X$ and consist  of norm one vectors. 
		Let $$y_\lambda=\sum_{n\in \N}\lambda^nx_n$$ for each $\lambda \in (0,1/2)$.
		We claim that $Y=\{y_\lambda: \lambda\in (0,1/2)\}$ satisfies the theorem.
		Let $H$ be a hyperplane $x^*\in X^*$ be the norm one nonzero linear bounded functional
		whose kernel is $H$.
		We have 
		$\limsup_{n\rightarrow \infty}\sqrt[n]{|x^*(x_n)|}
		\leq\sup_{n\in \N}\sqrt[n]{|x^*(x_n)|}\leq 1$ and so the formula
		$$f(\lambda)=\sum_{n\in \N}x^*(x_n)\lambda^n$$
		defines an analytic function on $(-1, 1)$. $f\equiv0$ on $(-1, 1)$ only
		if $x^*(x_n)=0$ for each $n\in \N$, which is not the case since $x^*$ is not the zero functional
		on $X$. By the properties of analytic functions $f$ cannot have infinitely many zeros in $(0, 1/2)$,
		which means that  
		$0=f(\lambda)=x^*(\sum_{n\in B}\lambda^n x_n)=x^*(y_\lambda)$ only for finitely many 
		$\lambda\in (0,1)$ as required.
	\end{proof}

	\begin{reptheorem}[\ref{main-separable}]
		Suppose that $X$ is a separable Banach space of dimension bigger than $1$. Then
		the following equalities hold:
		\begin{itemize}
			\item $\mathfrak{add}(X)=\omega_1$, 
			\item $\mathfrak{non}(X)=\omega_1$, 
			\item $\mathfrak{cov}(X)=\mathfrak{c}$,
			\item $\mathfrak{cof}(X)=\mathfrak{c}$.
		\end{itemize}
	\end{reptheorem}
	\begin{proof}
		$|\h(X)|\leq \mathfrak c$ if $X$ is separable as hyperplanes are determined by
		continuous functionals and such are determined by their values on a dense set.
		So by Proposition \ref{simple-cichon}  it is enough to prove that
		$\non(X)=\omega_1$ and $\cov(X)=\mathfrak c$.
		Let $Y$ be the set from Proposition \ref{klee} and $Y'\subseteq Y$ any set such that $|Y'|=\omega_1$. If $Y'$ is covered by countably many hyperplanes $\{H_n\}_{n\in \N}$, then there is $n\in \N$ for which $H_n$ contains infinite subset $Z\subseteq Y'$, so $H_n=\overline{\lin (Z)}=X$, which is contradiction. Hence $\non(X)=\omega_1$. 
		
		Assume now that $X$ is covered by $\kappa<\mathfrak c$ sets from $\I(X)$. Then $X$ is covered by $\kappa$ hyperplanes, so there is a hyperplane $H$ containing infinite subset of $Y$ and again we get contradiction. Hence $\cov(X)=\mathfrak c$. 
	\end{proof}
	
	Note that the first and last equations are also special cases of Propositions \ref{basic-add} and \ref{basic-cof}.
	
	\subsection{General Banach spaces}\label{general}
	
	\begin{proposition}\label{basic-add}
		Let $X$ be a Banach space of dimension bigger than $1$. Then 
		$$\add(X)=\omega_1.$$
	\end{proposition}
	\begin{proof} It is clear that $\add(X)\geq \omega_1$. If $f,g\in X^*$ are linearly independent, then the kernels of $f+\lambda g$  are different hyperplanes for different choices of $\lambda \in \R\backslash \{0\}$.
		So let $\F$ be any collection of $\omega_1$-many distinct hyperplanes. We have $\F\subseteq \I$.
		However $\bigcup\F\not\in \I$ because otherwise if $\{H_i: i\in \N\}\subseteq\h$ and $\bigcup\F\subseteq \bigcup_{i\in \N}{H_i}$,
		then for every $H\in \F$ we have $H=H_i$ for some $i\in \N$ by Proposition \ref{inclusion} which
		contradicts the fact that $\F$ is uncountable.
	\end{proof}
	
	\begin{proposition}\label{basic-cof}
		Let $X$ be a Banach space of dimension bigger than $1$. Then 
		$$\cof(X)=|X^*|.$$
	\end{proposition}
	\begin{proof}
		Let $\F$ be a cofinal family in $\I(X)$. Without losing generality we can assume that $\F$ consists of countable sums of hyperplanes. By Lemma \ref{inclusion} every set in $\F$ contains only countably many hyperplanes, so $|\F|\geq |X^*|$. Moreover $|\F|$ is not greater than cardinality of the family  of all countable sets of hyperplanes which is equal to $|X^*|^\omega=|X^*|$
		by Proposition \ref{card-banach}. Thus $|\F|=|X^*|$.
	\end{proof}
	
	\begin{proposition}\label{basic-non} Let $X$ be a Banach space of dimension bigger than $1$. 
		Then
		$$\dens(X)\leq \mathfrak{non}(X)\leq \cf([\dens(X)]^\omega).$$
		If $\cf(\dens(X))=\omega$, then $\dens(X)< \mathfrak{non}(X)$. 
	\end{proposition}
	\begin{proof}
		Assume that $Y\subseteq X$ and  $|Y|< \dens(X)$. Then $\overline{\lin(Y)}$ 
		is a proper subspace of $X$ and so it  is contained in some hyperplane and
		hence $Y\in \I$, so $\dens(X)\leq \mathfrak{non}(X)$. 
		
		Let $\{x_\alpha: \alpha<\dens(X)\}$ be a dense subset of $X$.
		Let $\mathcal F\subseteq [\dens(X)]^\omega$ be a family which is cofinal 
		in $[\dens(X)]^\omega$ and of cardinality $\cf([\dens(X)]^\omega)$. 
		By Proposition \ref{klee} for each $F\in \mathcal F$
		the subspace 
		$X_F=\overline{\lin\{x_\alpha: \alpha\in F\}}\subseteq X$  contains a subset
		$Y_F$ such that  $|Y_F|=\omega_1$ and it cannot be covered by countably many hyperplanes in $X_F$. Put 
		$Y=\bigcup_{F\in \mathcal F} Y_F$. We claim that $Y\not\in \I(X)$ and $|Y|=[\dens(X)]^\omega$.
		If $Y$ were covered by countably many hyperplanes $H_n$ of $X$,  there would be $F\in \mathcal F$ such that $H_n\cap X_F\not=X_F$ for all $n\in \N$ which is contradiction with the choice of $Y_F$. 
		Hence $Y\not\in\I(X)$. Also $|Y|=\omega_1\cdot|\mathcal F|
		=\omega_1\cdot \cf([\dens(X)]^\omega)=\cf([\dens(X)]^\omega)$ as $\dens(X)$ is uncountable.

		Now assume that $\cf(\dens(X))=\omega$. If $Y\subseteq X$ and $|Y|=\dens(X), Y=\bigcup_{n\in\N} Y_i$ with $|Y_i|<\dens(X)$, then every $Y_i$ is contained in some closed subspace of $X$ and hence in a hyperplane $H_i$
		for $i\in \N$. Thus $Y\in \I$.
	\end{proof}
	
	\begin{proposition}\label{basic-cov}
		Let $X$ be a Banach space of dimension bigger than $1$. Then 
		$$\omega_1\leq\cov(X)\leq \mathfrak c.$$
		In particular, under {\sf CH}, $\cov(X)=\omega_1$ for every nonseparable Banach space $X$.
		If $\cf(\dens(X))>\omega$, then $\mathfrak{cov}(X)\leq \cf(\dens(X))$. In particular
		if $\dens(X)=\omega_1$, then $\cov(X)=\omega_1$.
	\end{proposition}
	\begin{proof} Since $\I(X)$ is a $\sigma$-ideal, we have $\omega_1\leq\cov(X)$.
		Let $f,g\in X^*$ be linearly independent. Then for every $x\in X$ there are $(a,b)\in \R\backslash \{(0,0)\}$ such that $af(x)+bg(x)=0$. Thus the family 
		of hyperplanes $\{\ker af+bg : (a,b)\in \R^2\backslash \{(0,0)\} \}$ of cardinality $\mathfrak c$ covers $X$.
		
		Let $\kappa=\dens(X)$ and let  $\{x_\alpha: \alpha<\kappa\}$
		be a dense subset of $X$. Let $\kappa=\sup\{\alpha_\xi: \xi<\cf(\kappa)\}$.
		Let $X_\xi=\overline{\lin\{x_\alpha: \alpha<\alpha_\xi\}}$ for $\xi<\cf(\kappa)$.
		Each $X_\xi$ is  a proper subspace of $X$ since the density of $X$ is $\kappa>\alpha_\xi$.
		Also every element $x\in X$ is in the closure of a countable
		subset of $\{x_\alpha: \alpha<\kappa\}$, and so by the uncountable
		cofinality of $\kappa$ we conclude that $x\in X_\xi$ for some $\xi<\cf(\kappa)$.
	\end{proof}

	\section{Covering nonseparable Banach spaces with $\omega_1$ hyperplanes}\label{cov}
	
	By Proposition \ref{basic-cov} and Theorem \ref{main-separable} if we assume
	${\sf CH}$ we have $\cov(X)=\omega_1$ for all
	Banach spaces $X$. In this section we investigate whether $\cov(X)=\omega_1$
	may hold for all nonseparable Banach spaces without this assumption (Note that by Theorem \ref{main-separable}
	if {\sf CH} fails, then $\cov(X)>\omega_1$ for all separable Banach spaces).
	We prove that the value of $\cov$ is indeed $\omega_1$ for many classes of nonseparable Banach spaces
	(Propositions \ref{fund-biort}, \ref{ell1-inside}, \ref{scattered}) and that consistently it holds for all Banach spaces in the presence of diverse negations
	of {\sf CH} (Proposition \ref{con-cov}). 
	The deepest observations rely heavily on set-theoretic topological results of \cite{juhasz-szentmiklossy},
	\cite{dow-pavlov}, \cite{dow-countable-tightness}
	concerning small diagonal and countable tightness in compact Hausdorff spaces (Definition \ref{def-small-diagonal}).
	
	\begin{lemma}\label{cov-monotone} Suppose that $X, Y$ are Banach spaces and
		$T:X\rightarrow Y$ is a bounded linear operator whose range is dense in $Y$.
		Then $\cov(Y)\leq\cov(X)$.
	\end{lemma}
	\begin{proof}If $0\not=y^*\in Y^*$, then $T^*(y^*)\not=0$ because the range of $T$ is dense in $Y$, so
		a covering of $Y$ by hyperplanes induces a covering of $X$ by hyperplanes which is of the same cardinality
		which proves $\cov(X)\leq \cov(Y)$.
	\end{proof}
	
	\begin{lemma}\label{bound-of-cov} For every nonseparable Banach space $X$ there
		is a linear bounded operator $T: X\rightarrow \ell_\infty(\omega_1)$ with nonseparable range.
		In particular, all values of
		the cardinal characteristic $\cov$ on nonseparable Banach spaces are bounded  by the values
		on  nonseparable subspaces of  $\ell_\infty(\omega_1)$.
	\end{lemma}
	\begin{proof}  Every Banach space is isometric to a subspace of $C(K)\subseteq \ell_\infty(K)$,
		where $K=B_{X^*}$. So we may assume that $X\subseteq \ell_\infty(\kappa)$ for some uncountable cardinal
		$\kappa$. As $X$ is nonseparable, it contains an uncountable discrete set $D$. This fact is witnessed by
		the coordinates from some set $A\subseteq\kappa$ of cardinality $\omega_1$.  That is 
		there exist $\varepsilon>0$ such that for every distinct $d, d'\in D$ we have 
		is $|d(\alpha)-d'(\alpha)|>\varepsilon$ for some  $\alpha\in A$. Consider
		the restriction operator $R: X\rightarrow \ell_\infty(A)$. It is clear that the range is
		nonseparable by the choice of $A$. To
		conclude the last part of the lemma take any nonseparable Banach space $X$ and consider
		the operator $T$ as in the first part of the lemma and let $Y$ be 
		the closure of the range of $T$. Using Lemma \ref{cov-monotone}
		we conclude that $\cov(X)\leq \cov(Y)$.
	\end{proof}
	
	Let us now prove a simple but useful: 
	
	\begin{lemma}\label{equiv-cov-omega1} Let $X$ be a  Banach space. The following conditions are
		equivalent:
		\begin{enumerate}
			\item $\cov(X)=\omega_1$.
			\item $X$ is a union of $\omega_1$ hyperplanes.
			\item There is $A\subseteq X^*\setminus\{0\}$ of cardinality $\omega_1$
			such that for every $x\in X$ there is $x^*\in A$ such that $x^*(x)=0$.
			\item There is a bounded linear operator $T:X\rightarrow\ell_\infty(\omega_1)$
			such that 
			\begin{enumerate}
				\item for every $\alpha<\omega_1$ there is $x\in X$ such that $T(x)(\alpha)\not=0$.
				\item for every $x\in X$ there is $\alpha<\omega_1$ such that $T(x)(\alpha)=0$.
			\end{enumerate}
		\end{enumerate}
	\end{lemma}
	\begin{proof}
		The equivalence of the first three items is clear. Assume (3) and let us prove (4).
		Let $\{H_\alpha: \alpha<\omega_1\}$ be the hyperplanes that cover $X$ and
		let $x_\alpha^*\in X^*$ be such that $H_\alpha$ is the kernel of $x^*_\alpha$ and
		$\|x_\alpha^*\|=1$ for
		all $\alpha<\omega_1$. Let $T(x)(\alpha)=x_\alpha^*(x)$. Condition (a) follows from the fact that $x_\alpha^*\not=0$ and condition (b) from the fact that $H_\alpha$s cover $X$.

		Now assume (4) and let us prove (3). Condition (a) implies that $x_\alpha^*=T^*(\delta_{\alpha})$
		is a nonzero element of $X^*$, and so its kernel is a hyperplane. Condition (b) implies that
		the kernels of $x_\alpha^*$s cover $X$.
	\end{proof}

	\begin{proposition}\label{fund-biort} Let $X$ be a nonseparable Banach space.
		Each of the following sentences implies the next.
		\begin{enumerate}
			\item $X$ admits a fundamental biorthogonal system.
			\item There is a bounded linear operator $T: X\rightarrow c_0(\omega_1)$ with
			nonseparable range (i.e., $X$ is not half-pcc in the terminology of \cite{djp}).
			\item $\cov(X)=\omega_1$.
		\end{enumerate}
		In particular for every nonseparable WLD Banach space $X$ we have $\cov(X)=\omega_1$.
	\end{proposition}
	\begin{proof}
		Let $\{(x_\alpha, x^*_\alpha): \alpha<\kappa\}$ be a fundamental
		biorthogonal system.  Define $T: X\rightarrow \ell_\infty(\omega_1)$ 
		by $T(x)(\alpha)=x_\alpha^*(x)$. As $T(x_\alpha)=1_{\{\alpha\}}|\omega_1\in c_0(\omega_1)$
		and $X$ is the closure of the linear span of $\{x_\alpha: \alpha<\omega_1\}$
		we conclude that $T[X]\subseteq c_0(\omega_1)$. $T(x_\alpha)=1_{\{\alpha\}}$
		for $\alpha<\omega_1$ witnesses the fact that the range is nonseparable.
		
		Now assume (2). As the range of $T$ is nonseparable, by passing to
		an uncountable set of coordinates we may assume that for all
		$\alpha<\omega_1$ there is $x\in X$ such that $T(x)(\alpha)\not=0$.
		So item (4) of Lemma \ref{equiv-cov-omega1} is satisfied, and hence $\cov(X)=\omega_1$.
		
		To make the final observation use the fact that WLD
		Banach spaces admit fundamental biorthogonal systems e.g. by the results of \cite{zizler}.
	\end{proof}
	
	Note that the paper \cite{djp} contains many results on properties of Banach spaces $X$ which imply
	item (2) of Lemma \ref{fund-biort}, for example this happens when $X^*$ contains a nonmetrizable
	weakly compact subset. To obtain more Banach spaces $X$ satisfying $\cov(X)$ 
	we need some topological considerations. First recall the following:
	
	\begin{definition}\label{def-small-diagonal} Let $K$ be a compact Hausdorff space.
		\begin{enumerate}
			\item We say that $K$ has small diagonal if
			for every uncountable subset $A$ of $K^2\setminus\Delta(K)$ there
			is an uncountable $B\subseteq A$ whose closure is disjoint from $\Delta(K)$.
			\item We say that $K$ has countable tightness (is countably determined) if whenever $K\ni x\in\overline{A}$
			for $A\subseteq K$, then there is a countable $B\subseteq A$ such that $x\in\overline{B}$.
		\end{enumerate}
	\end{definition}
	
	In the following lemma the implication from (3) to (4) is the result of \cite{juhasz-szentmiklossy}.
	
	\begin{lemma}\label{small-diagonal} Suppose that $K$ is a compact Hausdorff space.
		Each of the following sentences implies the next.
		\begin{enumerate}
			\item For every $A\subseteq K$ of cardinality $\omega_1$ there is  a continuous $f: K\rightarrow \R$ such that $f|A$ is injective.
			\item For every $A=\{(x_\alpha, y_\alpha): \alpha<\omega_1\}\subseteq K^2\setminus\Delta(K)$ of 
			cardinality $\omega_1$ there  is  a continuous
			$f: K\rightarrow \R$ such that 
			$$\{\alpha: f(x_\alpha)\not=f( y_\alpha)\}$$
			is uncountable. 
			\item $K$ has small diagonal. 
			\item $K$ is countably tight.
		\end{enumerate}
	\end{lemma}
	\begin{proof}
		Assume (1).
		Let $A\subset K^2$ be a set of cardinality $\omega_1$ disjoint from the diagonal.
		Let $A=\{(x_\alpha,y_\alpha):\alpha<\omega_1\}$. 
		Put $L=\{x_\alpha, y_\alpha: \alpha<\omega_1\}$ and let $f:K\rightarrow \R$ be continuous and  $f|L$  injective. Then $f(x_\alpha)\not=f( y_\alpha)$ for each $\alpha<\omega_1$, so we obtain (2).
		
		Assume (2). Let $A\subset K^2$ be uncountable. 
		We may assume that $A$ is   of cardinality $\omega_1$ and so  $A=\{(x_\alpha,y_\alpha):\alpha<\omega_1\}$.
		Let  $f: K\rightarrow \R$  be continuous  and  such that 
		$$\{\alpha: f(x_\alpha)\not=f( y_\alpha)\}$$
		is uncountable. 
		Then $g:K^2\rightarrow \R$ defined by $g(x,y)=|f(x)-f(y)|$ is continuous and 
		$g|A$ is non-zero. Hence there exist $\varepsilon>0$ and an uncountable subset $A'\subseteq A$ such that $g(a)>\varepsilon$ for $a\in A'$. It follows that the closure of $A'$ is disjoint from diagonal as $g|\Delta(K)=0$ which completes the proof of (3). 
		
		For the last implication see the proof of \cite[Corollary 2.3]{juhasz-szentmiklossy}.
	\end{proof}
	
	It is easy to see that the one-point compactification of an uncountable discrete space is countably tight but does not have small diagonal, so (4) does not imply (3).  We do not know if the other implications reverse (cf. Question \ref{compact-omega1}).
	
	\begin{lemma}\label{dualball-diagonal} Let $X$ be a Banach space. Each of the following sentences implies the next.
		\begin{enumerate}
			\item The dual ball $B_{X^*}$ does not have small diagonal in the weak$^*$ topology.
			\item There is $\{x^*_\alpha: \alpha<\omega_1\}\subseteq B_{X^*}\setminus \{0\}$
			such that $\{\alpha: x^*_\alpha(x)\not=0\}$ is at most countable for each $x\in X$.
			\item There is $A\subseteq B_{X^*}$ of cardinality $\omega_1$ such that $\delta_x|A$ is not injective for each $x\in X$, where $\delta_x\in C(B_{X^*})$ is given by $\delta_x(x^*)=x^*(x)$.
			\item $\cov(X)=\omega_1$.
		\end{enumerate}
	\end{lemma}
	\begin{proof} Suppose (1). By the implication from (2) to (3)
		of Lemma \ref{small-diagonal} there
		is $A=\{(y_\alpha^*, z_\alpha^*): \alpha<\omega_1\}\subseteq B_{X^*}^2\setminus\Delta(B_{X^*})$ of 
		cardinality $\omega_1$ such that for every continuous
		$f: K\rightarrow \R$ the set 
		$\{\alpha: f(y_\alpha^*)\not=f( z_\alpha^*)\}$
		is countable.  Of course $x\in X$ defines a continuous function on the dual ball
		in the weak$^*$ topology so $\{\alpha: (y_\alpha^*-z_\alpha^*)(x)\}$ is countable
		for all $x\in X$. So we put $x_\alpha^*=y_\alpha^*-z^*_\alpha$ and we obtain (2).
		
		Assume (2). Put $A=\{x^*_\alpha: \alpha<\omega_1\}$. Then for each $x\in X$ image of $\delta_x|A$ is countable, so $\delta_x|A$ is not injective since $|A|> \omega$. 
		
		Now suppose (3). Consider $\mu_\alpha=x_\alpha-y_\alpha\in X^*$, where
		$\{\{x_\alpha, y_\alpha\}: \alpha<\omega_1\}=[A]^2$. Then the kernels
		of $\mu_\alpha$'s cover $X$.
	\end{proof}
	
	Since it is consistent that all nonmetrizable compact spaces do not have small diagonals, it is also consistent that sentences (1)-(4) from Lemma \ref{dualball-diagonal} are equivalent in the class of nonseparable Banach spaces. It is still an open question whether this holds in {\sf ZFC}.
	
	\begin{proposition}\label{ell1-inside} If $X$ is a Banach space which contains an isomorphic copy 
		of $\ell_1(\omega_1)$, then $\cov(X)=\omega_1$. In particular this holds for
		any space which contains $\ell_\infty$ like $\ell_\infty(\kappa)$, $L_\infty(\{0,1\}^\kappa)$,
		$\ell_\infty/c_0$ etc.
	\end{proposition}
	\begin{proof} By the main result of \cite{talagrand}, if a Banach space $X$ contains  $\ell_1(\omega_1)$,
		then there is a continuous surjection $\Phi: B_{X^*}\rightarrow [0,1]^{\omega_1}$, 
		where $B_{X^*}$ is considered with the weak$^*$ topology. As countable tightness is preserved
		by continuous map and $[0,1]^{\omega_1}$ is not countably tight (consider
		$1_{[0,\omega_1)}\in \overline{\{1_{[0,\alpha)}: \alpha<\omega_1\}}$) we conclude that
		$B_{X^*}$ is considered with the weak$^*$ topology is not countably tight.
		By Lemma \ref{small-diagonal} $B_{X^*}$ does not have small diagonal, and so by Lemma \ref{dualball-diagonal}
		we conclude that $\cov(X)=\omega_1$.
	\end{proof}

	\begin{proposition}\label{scattered} If $K$ is compact nonmetrizable and scattered, then $\cov(C(K))=\omega_1$.
	\end{proposition}
	\begin{proof} $K$ must be uncountable.
		Let $A\subseteq K$ be any subset of cardinality $\omega_1$.
		As a continuous image of a scattered compact space
		is scattered compact we conclude that 
		for any continuous
		$f:K\rightarrow \R$ the image of $f$ is countable
		and so $f|A$ is not injective which implies that $\delta_f|\{\delta_x:x\in A\}$ is not injective. Hence $C(K)$ satisfy condition (3) of Lemma \ref{dualball-diagonal}.
	\end{proof}

	\begin{proposition}\label{con-cov}$ $\begin{enumerate}
			\item {\sf PFA} implies that every nonseparable Banach space $X$ satisfies $\cov(X)=\omega_1$.
			\item It is consistent with any possible size of the continuum, that every nonseparable Banach space $X$ satisfies $\cov(X)=\omega_1$.
		\end{enumerate}
	\end{proposition}
	\begin{proof} It is shown in \cite{dow-pavlov} that assuming {\sf PFA}
		every compact Hausdorff space with small diagonal is metrizable.  So by Lemma \ref{dualball-diagonal}
		we conclude that $\cov(X)=\omega_1$ for every nonseparable Banach space $X$ under {\sf PFA}.
		Similarly Theorem 5.8 from \cite{dow-countable-tightness} shows that
		it is consistent with any possible size of the continuum (in models obtained from {\sf CH} model by adding Cohen reals) that each compact space with countable tightness has small diagonal if and only if it is metrizable. However, non-countably-tight compact spaces cannot have a small  diagonal (in ZFC) by the result of \cite{juhasz-szentmiklossy} that is the implication from (3) to (4) in Lemma \ref{small-diagonal}. 
		So by Lemma \ref{dualball-diagonal}
		we conclude that $\cov(X)=\omega_1$ for every nonseparable Banach space $X$ in these models as well.
	\end{proof}

	\section{Covering small subsets of Banach spaces  by countably many hyperplanes}\label{non}
	
	By Proposition \ref{basic-non} if $X$ is a Banach space of dimension bigger than 1
	the value of $\non(X)$
	(i.e.,  the minimal cardinality of a set not covered by countably many hyperplanes)
	is in the interval $[\dens(X), \cf([\dens(X)]^\omega)]$ if $\cf(\dens(X))$ is uncountable
	and is in the interval $[\dens(X)^+, \cf([\dens(X)]^\omega)]$ if  $\cf(\dens(X))$ is countable. 
	As we will see below, just purely set-theoretic known results imply that 
	under many assumptions these intervals reduce to singletons and so the values
	of $\non(X)$ are completely determined by $\dens(X)$. 
	It remains open, however, if
	$\non(X)=\cf([\dens(X)]^\omega)]$ for every nonseparable Banach space without any extra
	set-theoretic assumptions.

	\begin{proposition}\label{non-omegan}
		If $X$ is a Banach space with density $\omega_n$ for $n\in \N \backslash \{0\}$,
		then $\non(X)=\dens(X)=\cf([\dens(X)]^\omega)$.
	\end{proposition}
	\begin{proof}
		By induction on $n\in \N$, using the decomposition of $\omega_n$ into smaller ordinals
		one proves that $\cf([\omega_n]^\omega)=\omega_n$. Now Proposition \ref{basic-non} implies
		that $\non(X)= \omega_n$.
	\end{proof}
	
	\begin{proposition}\label{kappa-functionals}
		Let $X$ be a Banach space of density $\kappa$ and dimension bigger than $1$. Suppose that there are functionals $\{x^*_\alpha:\alpha<\kappa\}\subseteq X^*$ such that for every $x\in X$ the set $Z_x=\{\alpha: x^*_\alpha(x)\neq 0\}$ is countable. Then $\non(X)=\cf([\kappa]^\omega)$. 
	\end{proposition}
	\begin{proof}
		Let $\lambda<\cf([\kappa]^\omega)$ and $Y=\{x_\alpha: \alpha<\lambda\}\subseteq X$. By the assumption the family $\mathcal Z=\{Z_x:x\in Y\}$  is not cofinal in $[\kappa]^\omega$. Pick $A\in [\kappa]^\omega$, which is not included in any element of $\mathcal Z$. Then for every $x\in Y$ there is $\alpha\in A$ such that $x^*_\alpha(x)=0$, so $x$ is in kernel of $x^*_\alpha$. Thus kernels of $x^*_\alpha$s
		for $\alpha\in A$  cover $Y$, which proves that $\non(X)\geq \cf([\kappa]^\omega)$. The inequality $\non(X)\leq \cf([\kappa]^\omega)$ is true by Proposition
		\ref{basic-non}.
	\end{proof}
	
	\begin{proposition}\label{fund-non}
		If a Banach space $X$ of density $\kappa$ and dimension bigger than $1$ admits a fundamental biorthogonal system, then $\non(X)=\cf([\kappa]^\omega)$.
	\end{proposition}
	\begin{proof}
		Let $\{x_\alpha, x_\alpha^*\}_{\alpha<\kappa}$ be a fundamental biorthogonal system. For every $x$ pick a countable set $L_x\subset \kappa$ such that $x\in \overline{\lin \{x_\alpha: \alpha\in L_x\}}$. Then $Z_x=\{\alpha: x^*_\alpha(x)\neq 0\}\subseteq L_x$, so $Z_x$ is also countable. Hence $x_\alpha^*$s satisfy conditions of Proposition \ref{kappa-functionals}.
	\end{proof}
	
	\begin{proposition}\label{non-gch-mm}
		Assume that $\kappa^\omega=\kappa$ for all regular $\kappa>\omega_\omega$. Let $X$ be a Banach space of dimension bigger than $1$. Then 
		\begin{enumerate}
			\item If $\cf(\dens(X))=\omega$, then $\non(X)=\dens(X)^+$,
			\item If $\cf(\dens(X))>\omega$, then $\non(X)=\dens(X)$.
		\end{enumerate}
		In particular the above equations hold under {\sf GCH} or {\sf MM}.
	\end{proposition}
	\begin{proof}
		By assumption we have $\cf([\kappa]^\omega)=\kappa$ for all regular $\kappa>\omega_\omega$. If $\kappa$ is singular of uncountable cofinality then $\kappa^\omega= \Sigma_{\mu<\kappa} \ \mu^\omega=\Sigma_{\mu<\kappa} \ \mu = \kappa$ so $\cf([\kappa]^\omega)=\kappa$. If $\cf(\kappa)=\omega$ then $\kappa^\omega>\kappa$ and $\kappa^\omega\leq (\kappa^+)^\omega=\kappa^+$ so $\cf([\kappa]^\omega)\leq\kappa^+$. Hence for $\dens(X)\geq\omega_\omega$ the equalities (1) and (2) follow from Proposition \ref{basic-non}. The case when $\dens(X)= \omega_n$ for some $n\in\N$ is covered by Proposition \ref{non-omegan}.
		
		For a limit cardinal $\kappa$ of uncountable cofinality under {\sf GCH} we have
		$$\kappa^\omega=\Sigma_{\lambda<\kappa}\lambda^\omega\leq \Sigma_{\lambda<\kappa}\lambda^+
		\leq \kappa^2=\kappa$$
		so $\kappa^\omega=\kappa$. For successor cardinals we have $(\kappa^+)^\omega=\kappa^\omega\kappa^+\leq (\kappa^+)^2=\kappa^+$.
		
		If {\sf MM} holds, then by Theorem 37.13 of \cite{jech} we have $\kappa^\omega=\kappa^{\omega_1}=\kappa$ for each regular $\kappa>\omega_1$.
		
	\end{proof}

	\begin{proposition}\label{non-anyc} It is consistent with any possible 
		size of the continuum that for every  Banach space $X$ of dimension bigger than $1$ we have
		\begin{enumerate}
			\item If $\cf(\dens(X))=\omega$, then $\non(X)=\dens(X)^+$,
			\item If $\cf(\dens(X))>\omega$, then $\non(X)=\dens(X)$.
		\end{enumerate}
	\end{proposition}
	\begin{proof} Start with a model $V$ of {\sf GCH} and increase the continuum
		using a c.c.c. forcing (e.g. add Cohen reals). The cardinals and their cofinalities do not change.
		Moreover $[\kappa]^\omega\cap V$ is cofinal in $[\kappa]^\omega$ as any
		countable set of ordinals in a c.c.c. extension is included in a countable set in the ground
		model, so the calculations from the proof of Proposition \ref{non-gch-mm} remain true.
	\end{proof}

	\begin{proposition}\label{measurable} For every  Banach space $X$ of dimension bigger than $1$ we have
		\begin{enumerate}
			\item If $\cf(\dens(X))=\omega$, then $\non(X)=\dens(X)^+$,
			\item If $\cf(\dens(X))>\omega$, then $\non(X)=\dens(X)$,
		\end{enumerate}
		unless  there is a measurable cardinal in an inner model.
	\end{proposition}
	\begin{proof} If there is no  measurable cardinal in an inner model, then there is an
		inner model $M$ which satisfies {\sf GCH} and satisfies the covering lemma i.e.,
		$[\kappa]^{\omega_1}\cap M$ is cofinal in $[\kappa]^{\omega_1}$ for each cardinal $\kappa$ 
		(see \cite{dodd-jensen}).
		This implies that $[\kappa]^{\omega}\cap M$ is cofinal in $[\kappa]^{\omega}$ 
		for each cardinal $\kappa$. So since $M$ satisfies {\sf GCH},
		Proposition \ref{non-gch-mm} implies the theorem. (For a similar argument cf. the proof of Theorem 13.3 (d) in \cite{rinot}.)
	\end{proof}
	
	Recall that assuming the existence of a suitably large cardinal the consistency
	of $2^{\omega_n}<\omega_\omega$ and $2^{\omega_\omega}=\omega_{\omega+k}$ for any $n\in \N$ and $k>1$ was proved in
	\cite{magidor} (this problem was also considered with weaker assumptions in \cite{golshani}).
	In this case $\cf([\omega_{\omega}]^\omega)=\omega_{\omega+k}$ because
	$[\omega_{\omega}]^\omega= \bigcup\{[A]^\omega: A\in \F\}$ for any cofinal
	family in $[\omega_{\omega}]^\omega$ and $|\bigcup\{[A]^\omega: A\in \F\}|\leq \mathfrak{c}\cdot|\F|=|\F|$
	as $\mathfrak c<\omega_\omega$. It follows that $\cf([\omega_{\omega+m}]^\omega)\geq\omega_{\omega+k}$
	for $0\leq m<k$. So not only the existence of Banach spaces of density $\omega_\omega$  which assume the value of $\non$ smaller than
	in Propositions \ref{non-gch-mm}, \ref{non-anyc} if $k\geq 3$ but also
	of a regular density $\omega_{\omega+1}$ is not excluded by cardinal arithmetic in the considered model.

	\section{Final remarks}\label{final}
	
	\subsection{Densities of quotients of Banach spaces}
	
	The famous Separable Quotient Problem asks if every infinite dimensional Banach space
	has a separable infinite dimensional quotient. In the direction of bounding
	the densities of quotients of Banach spaces, one can easily prove that every Banach $X$ space has 
	a infinite dimensional quotient whose density is not bigger then $\mathfrak c$. 
	In this light the following is natural to ask:

	\begin{question}\label{question-quotients} $ $
		Is it true in ZFC that every nonseparable Banach space has a  quotient of density $\omega_1$?
	\end{question}

	By Propositions \ref{basic-cov}  and \ref{cov-monotone}
	the positive answer to question \ref{question-quotients}  would imply that
	$\cov(X)=\omega_1$  for every Banach space $X$. It would also imply that the Separable Quotient Problem
	consistently has  positive answer since it is proved in \cite{stevo-biort} that it is consistent that
	all Banach spaces of density $\omega_1$ have infinite dimensional separable quotients. 
	In fact, for this it would be enough to obtain the consistency of the positive answer to Question \ref{question-quotients}  with the additional set-theoretic assumptions of \cite{stevo-biort}, like the {\sf PFA}.

	\subsection{Banach spaces with no fundamental biorthogonal systems}
	Theorems \ref{main-cov} and \ref{main-non} determine the values of $\cov$ and $\non$ for
	Banach spaces admitting fundamental biorthogonal systems. So looking for
	spaces witnessing different values of $\cov$ or $\non$ we should understand better
	spaces not admitting such systems. The first and classical example of such a space  is the subspace 
	$\ell_\infty^c(\mathfrak c^+)$ of $\ell_\infty(\mathfrak c^+)$ consisting
	of elements with countable supports (\cite{godun-kadec}, \cite{plichko-fund}). However it contains a copy of $\ell_\infty$ and so
	$\ell_1(\omega_1)$ so $\cov(\ell_\infty^c(\lambda))=\omega_1$ for any infinite $\lambda$
	by Theorem \ref{main-cov}. Moreover
	Proposition \ref{kappa-functionals} implies that $\non(\ell_\infty^c(\mathfrak \lambda))=\cf([\lambda]^\omega)$
	for any $\lambda>\mathfrak c^+$ as $\dens(\ell_\infty^c(\mathfrak \lambda))=\lambda$ in such a case.
	Other reason for not admitting a fundamental biorthogonal system in a nonseparable space is not admitting
	any uncountable biorthogonal system: 
	The Kunen line and the examples of \cite{jordi-stevo}, \cite{rolewicz}, \cite{christina} have all
	density $\omega_1$, so they have $\cov=\omega_1$ by Proposition \ref{basic-cov}.
	The only known Banach space of density bigger than $\omega_1$ with no uncountable biorthogonal systems
	is that of \cite{christina-vtt}. However it is of the form $C(K)$ with $K$ scattered so
	Theorem \ref{main-cov} implies  that $\cov=\omega_1$. It also has density $\omega_2$,
	so Theorem \ref{main-non} implies that its $\non$ is $\omega_2=\cf([\omega_2]^\omega)$.

	\subsection{A question on compact Hausdorff spaces} By Lemma \ref{dualball-diagonal}
	positive answer to the following question would imply that
	$\cov(X)=\omega_1$ for every nonseparable Banach space $X$:
	
	\begin{question}\label{compact-omega1} Is it provable that
		every nonmetrizable compact Hausdorff space $K$ admits a subspace
		$L\subseteq K$ of cardinality $\omega_1$  such that for no $f\in C(K)$ the restriction $f|L$ is injective?
	\end{question}
	
	Recall that it was proved in \cite{dow-compact} that every nonmetrizable compact Hausdorff space admits
	a subspace of size $\omega_1$ which is nonmetrizable. Moreover the
	above result and  Proposition 11 of \cite{dow-pavlov} imply that every nonmetrizable compact Hausdorff space $K$ admits a subspace $L\subseteq K$ of cardinality $\omega_1$  such that for no $f\in C(K)$ 
	we have $f^{-1}[\{f(x)\}]=\{x\}$ for all $x\in L$.

	\bibliographystyle{amsplain}
	
\end{document}